\documentclass[reqno,12pt]{amsart} 
\usepackage{vmargin}
\setpapersize{A4}

\usepackage{amsmath} 

\usepackage{amsfonts}   

\usepackage{amssymb}      

\usepackage{eufrak}      

\usepackage{amscd}      
\usepackage{amsthm}      

\usepackage{tikz,amsmath}
\usetikzlibrary{arrows}

\usepackage{epsfig}      

\usepackage{amstext}      

\usepackage{graphicx}

\usepackage[all,line,dvips]{xy} 
\CompileMatrices 

\newcommand{\id}{\operatorname{id}} 
\newcommand{\Aut}{\operatorname{Aut}}

\newcommand{\co}{\operatorname{co}}

\newcommand{\Per}{\operatorname{Per}}

 \newcommand{\supp}{\operatorname{supp}}

\newcommand{\Is}{\operatorname{Is}}

\newcommand{\ev}{\operatorname{ev}}

\newcommand{\Ro}{\operatorname{RO}}

   \theoremstyle{plain}
   \newtheorem{thm}{Theorem}[section]
   \newtheorem{prop}[thm]{Proposition}
   \newtheorem{lemma}[thm]{Lemma}  
   \newtheorem{cor}[thm]{Corollary}
   \theoremstyle{definition}

   \newtheorem{example}[thm]{Example}
   \theoremstyle{remark}
   
   \newtheorem{remark}[thm]{Remark}

\newtheorem{assumption}[thm]{Assumption}

\usepackage{mdframed,xcolor}

\definecolor{mybgcolor}{gray}{0.8}
\definecolor{myframecolor}{rgb}{.647,.129,.149}

\mdfdefinestyle{mystyle}{
  usetwoside=false,
  skipabove=0.6em plus 0.8em minus 0.2em,
  skipbelow=0.6em plus 0.8em minus 0.2em,
  innerleftmargin=.25em,
  innerrightmargin=0.25em,
  innertopmargin=0.25em,
  innerbottommargin=0.25em,
  leftmargin=-.75em,
  rightmargin=-0em,
  topline=false,
  rightline=false,
  bottomline=false,
  leftline=false,
  backgroundcolor=mybgcolor,
  splittopskip=0.75em,
  splitbottomskip=0.25em,
  innerleftmargin=0.5em,
  leftline=true,
  linecolor=myframecolor,
  linewidth=0.25em,
}

\newmdenv[style=mystyle]{important}

\usepackage{lipsum}

   \numberwithin{equation}{section}

        \date{\today}

\title[Positive eigenvectors]{On the positive eigenvalues and
  eigenvectors of a non-negative matrix}  
\author{Klaus Thomsen}

\date{\today}

\email{matkt@imf.au.dk}
\address{Institut for Matematik, Aarhus University, Ny Munkegade, 8000 Aarhus C, Denmark}

\begin{document}

\maketitle

\section{Introduction}

Recent work on KMS states and weights on graph $C^*$-algebras has
uncovered an intimate relation to positive eigenvalues and
eigenvectors for non-negative matrices naturally associated to the
graph and the one-parameter action. Specifically, for the gauge action on a graph
$C^*$-algebra there is a non-negative matrix $B$ over the vertexes $V$
in the graph, such that KMS weights corresponding to the inverse temperature $\beta
\in \mathbb R$ are in bijective correspondence with the non-zero non-negative
vectors $\xi$ with the properties that
\begin{equation}\label{eq1}
\sum_{w \in V} B_{vw}\xi_w \leq e^{\beta} \xi_v
\end{equation}
for all vertexes $v$, and 
\begin{equation}\label{eq2}
\sum_{w \in V} B_{vw}\xi_w = e^{\beta} \xi_v
\end{equation}
when $v$ is not a sink and does not
emit infinitely many edges. If a state, rather than a weight is sought
for, one should in addition
insist that $\sum_v \xi_v = 1$. The same kind of equations, but with
different matrices determine also the
gauge invariant KMS weights and states for more general actions
on graph $C^*$-algebras. See \cite{aHLRS},\cite{CL},\cite{Th}.

Finding the solutions to (\ref{eq1}) and (\ref{eq2}) is in general a
highly non-trivial task. The literature on positive eigenvalues and eigenfunctions of a
non-negative matrix is enormous, and for finite graphs there are in
fact results available that can be used to determine the possible
values of $\beta$ and for each $\beta$ get a description of the
corresponding vectors $\xi$, albeit with some additional work, cf. \cite{CT}. For infinite matrices this
is no longer the case - very far from. The most fundamental questions
prompted by the connection to
KMS states and weights are those that I guess come to the mind of any mathematician: 
\begin{enumerate}
\item[a)] For which $\beta$ are there non-zero non-negative solutions
  to the equations ?
\item[b)] How does the structure of the solutions vary with $\beta$ ? 
\end{enumerate} 
From the theory of countable state Markov chains it is known that
these questions are often extremely hard to answer already when the
matrix is irreducible and stochastic, but also that the
structure of the solutions can be very rich and interesting.  

The concern here is that a general setting for an approach to
the above problem is missing, although the theory of harmonic and
super-harmonic functions of countable state
Markov chains comes close. It is the purpose with the present paper to
provide a framework for the work on the problem when the graph
$C^*$-algebra is simple, or more precisely when the graph is
cofinal. As far as I know, no one has developed the theory in this setting. It involves a possibly infinite matrix $B$ as above for
which the underlying directed graph need not be strongly connected,
and what is sought are neither exactly the harmonic
functions of $e^{-\beta}B$ nor exactly the super-harmonic functions. But it is something in between and the cofinality of the
associated graph is a property which can substitute for the often
assumed strong connectivity. What I show is how the known methods,
which typically deal with harmonic or super-harmonic functions (or
vectors) of a non-negative matrix, often stochastic or sub-stochastic for which the underlying digraph is strongly
connected, can be modified to the yield the desired framework. As a
consequence the paper is expository because the ideas I present are
known. The purpose is to show how the tools must be arranged
in order to address the problem above. When the graph is strongly
connected with finite out-degree at every vertex, everything I present can be
obtained from the theory of countable state Markov chains in
combination with the work of Vere-Jones, \cite{V}, although the
translation may not be straightforward. In the more general
cofinal case, and in the presence of infinite emitters, some
non-trivial adjustments to the methods developed for Markov chains
must be performed, and rather than describing first the Markov chain
results and then the adjustments, I have chosen to give a
self-contained account, requiring no knowledge of Markov chains or
random walks. I make no claim of originality for the underlying ideas, but
I hope that the presentation will be useful for mathematicians interested
in KMS weights and states on $C^*$-algebras. Needless to say, I would
be happy if workers from other fields of mathematics also find it
worthwhile. I have kept the list of references to an absolute minimum by quoting
only my own sources. The reason
for this is that I am unable
to point to the original sources of the ideas presented, not for lack
of good will, but out of ignorance. I therefore choose to follow the
principle 'none mentioned, none forgotten'. I apologise to
anyone offended by this.

Once the the right setup is in place it is easy to begin to harvest results
from the theory of Markov chains. As an illustration of this I use in the
final section the
theorem on convergence to the boundary for a countable state Markov
chain to obtain the general and abstract description of
extremal solutions to equations (\ref{eq1}) and (\ref{eq2}) in the
cofinal case.

\section{The setting}\label{the setting}

Let $V$ be a countable set and 
$$
V \times V \ni (v,w) \mapsto A_{vw} \in [0,\infty)
$$ 
a
non-negative matrix $A$ over $V$. We can then consider the directed graph
$G$ with vertexes $V$ such that there is an arrow from $v \in V$ to
$w \in V$ if and only if $A_{vw} \neq 0$. Let $E$ denote the set of
edges (or arrows) in $G$. When $\mu$ is an edge, or more generally a finite path
in $G$, we denote by $s(\mu)$ and $r(\mu)$ its initial and terminal
vertex, respectively. Let $V_{\infty}$ denote the union of the sinks
and the infinite emitters in $V$, i.e.
$$
V_{\infty} = \left\{ v \in V : \ \# s^{-1}(v) \in \{0,\infty\} \
\right\} .
$$ 
A subset $H \subseteq V$ is \emph{hereditary} when $e \in E, s(e) \in
H \Rightarrow r(e) \in H$, and \emph{saturated} when
$$
v \in V \backslash V_{\infty}, \ r(s^{-1}(v)) \subseteq H \
\Rightarrow \ v \in H .
$$ 
We assume that $A$ is \emph{cofinal} in the sense that $V$ does not
contain any subsets that are both hereditary and saturated other than
$\emptyset$ and $V$.

\begin{lemma}\label{r1} Assume that $A$ is cofinal. Let $e_1e_2e_3 \cdots $ be an infinite
  path in $G$. For every $v \in V$ there is an $i \in \mathbb N$ and
  a finite path $\mu$ in $G$ such that $s(\mu) = v$ and $r(\mu) =
  s(e_i)$.
\end{lemma}
\begin{proof} The set of vertexes $v$ which do not have the stated
  property is hereditary and saturated, and it is not all of $V$ since
  it does not contain $s(e_1)$. It must therefore be empty.
\end{proof}

A vertex $v\in V$ is \emph{non-wandering} when there is a finite path
$\mu$ in $G$ such that $v = s(\mu) = r(\mu)$. We denote by $NW$ the
set of non-wandering vertexes in $V$. The \emph{non-wandering subgraph} of $G$ is the
subgraph $G^{NW}$ consisting of the vertexes $NW$ and the edges emitted from
any of its elements. It follows that $G^{NW}$ is \emph{strongly
  connected} in the sense that for any pair $v,w \in NW$ there is a finite
path $\mu$ in $G$ such that $s(\mu) = v$ and $r(\mu) = w$:

 \begin{lemma}\label{b71} $NW$ is a (possibly empty) hereditary subset of $V$ and
  the graph $G^{NW}$ is strongly connected.
\end{lemma}
\begin{proof} When $v \in NW$ there is an infinite path in $G$ which
  visits $v$ infinitely often. So when $e \in E$ and $s(e) \in NW$ the cofinality
  of $G$ ensures that there is a path $\mu$ in $G$ connecting $r(e)$ to
  $s(e)$ by Lemma \ref{r1}. Then $e\mu $ is a loop in $G$ containing $r(e)$, proving
  that $r(e) \in NW$, and hence that $NW$ is hereditary. The proof
  that $G^{NW}$ is strongly connected is similar. 
\end{proof}

Let $\beta \in \mathbb R$.
We are
here looking for maps (or vectors) $\xi : V \to
[0,\infty)$ such that
\begin{equation}\label{f20}
\sum_{w \in V} A_{vw}\xi_w = e^{\beta} \xi_v
\end{equation}
for all $v \in V \backslash V_{\infty}$ and
\begin{equation}\label{f20b}
\sum_{w \in V} A_{vw}\xi_w  \leq e^{\beta} \xi_v
\end{equation}
for $v \in V_{\infty}$. We say then that
$\xi$ is \emph{almost $\beta$-harmonic} for $A$. When (\ref{f20}) holds for all $v\in V$, and not only for $v
\in V \backslash V_{\infty}$, we say that $\xi$ is \emph{$\beta$-harmonic} for $A$.

\begin{lemma}\label{b65} Let $H$ be a non-empty hereditary subset of $V$ and
  $\eta : H \to [0,\infty)$ a function such that (\ref{f20}) holds for all
  $v \in H \backslash V_{\infty}$ and (\ref{f20b}) holds for all $v
  \in H \cap V_{\infty}$. There is a unique almost
  $\beta$-harmonic vector $\xi$ such that
$\xi_v = \eta_v$ for all $v \in H$.
\end{lemma}
\begin{proof} Set 
$$
H_1 = \left\{ v \in V  \backslash V_{\infty} : \ A_{vw} \neq 0 \ \Rightarrow \ w \in
  H \right\} \cup H.
$$
Then $H_1$ is hereditary, contains $H$ and there is a unique extension of $\eta$ to
$H_1$ given by the condition that $e^{\beta} \eta_v = \sum_{w\in V} A_{vw}\eta_w$ for all $v
\in H_1 \backslash H$. Continuing by induction we get a sequence $H \subseteq H_1
\subseteq H_2 \subseteq H_3 \subseteq \cdots$ of subsets of $V$ and a unique extension $\xi$
of $\eta$ to $\bigcup_n H_n$. This completes the proof since
$\bigcup_n H_n$ is hereditary and saturated, and hence equal to
$V$. 
\end{proof}

\begin{lemma}\label{r15} Assume that $G$ contains a sink (i.e. $A$
  contains a zero row). It follows that for all $\beta \in \mathbb R$
  there is a non-zero almost $\beta$-harmonic vector, unique up to
  multiplication by constants.
\end{lemma}
\begin{proof} Note that a sink $s$ constitutes a hereditary subset in
  itself. It follows therefore from Lemma \ref{b65} that an almost
  $\beta$-harmonic vector is determined by its value at $s$. To show
  that there exists a non-zero almost $\beta$-harmonic vector, set
  $\eta_s=1$ and apply Lemma \ref{b65}.
\end{proof}

\begin{lemma}\label{b16} Let $\xi$ be a non-zero almost $\beta$-harmonic
  vector. It follows that $\xi_v > 0$ for all $v \in
  V$.
\end{lemma}
\begin{proof} The set $\left\{v \in V  : \ \xi_v = 0  \right\}$ is
  hereditary and saturated. Since $\xi$ is not zero the set is not all
  of $V$, and it must therefore
  be empty. 
\end{proof}

We define the matrices $A^n, n = 0,1,2, \dots$, recursively such that
$A^0 = I$, where $I$ is the identity matrix,
$$
I_{vw} = \begin{cases} 1 & \ \text{when} \ v = w \\ 0 & \
  \text{otherwise}, \end{cases}
$$
$A^1 = A$, and
$$
A^{n+1}_{vw} = \sum_{u \in V} A_{vu}A^n_{uw}
$$
when $n \geq 1$. While the entries in $A$ are finite by assumption,
this need not be the case for $A^n$, but since (\ref{f20}) and (\ref{f20b}) imply
that
\begin{equation}\label{r11}
\sum_{w \in V} A^n_{vw}\xi_w \leq e^{n\beta} \xi_v
\end{equation}
for all $v,n$, the following conclusion follows from Lemma \ref{b16}.

\begin{lemma}\label{r12} Assume that there is a non-zero  almost $\beta$-harmonic vector for $A$. It follows that $A^n_{vw} <\infty$
  for  all $n\in \mathbb N$ and all $v,w \in V$.
\end{lemma}

In the following we therefore assume that all powers of $A$ are
finite. Since Lemma \ref{r15} contains all the information we seek
when there is a sink present, there is also nothing lost by assuming that there are no sinks in $G$. To summarise we assume in
the following that 
\begin{enumerate}
\item[i)] $G$ is cofinal,
\item[ii)] that there are no sinks in $G$ (equivalently, there are no
  zero rows in $A$), and
\item[iii)]  that $A^n_{vw} <\infty$
  for  all $n\in \mathbb N$ and all $v,w \in V$.
\end{enumerate}

In particular, from now on $V_{\infty}$ consists of the infinite
emitters in $G$, corresponding to rows in $A$ with infinitely many
non-zero entries.

\begin{lemma}\label{r12}  No vertex $v \in V \backslash NW$ is an
  infinite emitter, i.e. $V_{\infty} \subseteq NW$.
\end{lemma}
\begin{proof} Let $v \in V$ be an infinite emitter. Set
$$
A = \left\{w \in V: \ \text{there is a finite path $\mu$ in $G$ such
    that} \ w = s(\mu) \ \text{and} \ \ r(\mu) = v \right\} \cup \{v\}.
$$ 
Since $V \backslash A$ is hereditary and saturated, it follows
that $A = V$. In particular, $r(s^{-1}(v)) \subseteq A$, which implies
that $v\in NW$.    
\end{proof}

\begin{lemma}\label{i5} Let $\beta \in \mathbb R$. Assume that
  there are vertexes $v_0,w_0 \in V$ such that $\sum_{n=0}^{\infty}
  A^n_{v_0w_0} e^{-n\beta} = \infty$. Then
$\sum_{n=0}^{\infty} A^n_{vw_0} e^{-n\beta} = \infty$
for all vertexes $v \in V$.
\end{lemma}
\begin{proof} Set
$$
\mathcal C = \left\{ v \in V : \ \sum_{n=0}^{\infty} A^n_{vw_0}
  e^{-n\beta} < \infty \right\} .
$$
The equality
\begin{equation}\label{i2}
\sum_{u \in V} A_{vu} \sum_{n=0}^N A^n_{uw_0}e^{-n\beta} = e^{\beta}
\sum_{n=0}^{N+1} A^n_{vw_0}e^{-n\beta} - e^{\beta} I_{vw_0}
\end{equation}
shows that $\mathcal C$ is hereditary and saturated. Since $v_0 \notin
\mathcal C$ it follows that $\mathcal C = \emptyset$.
 
\end{proof}

When $NW$ is not empty we take an element $v \in NW$ and set
$$
\beta_0 = \log \left( \limsup_n \ (A^n_{vv})^{\frac{1}{n}} \right) 
$$
with the convention that $\log \infty = \infty$. Since $G^{NW}$ is strongly connected by Lemma \ref{b71} the value $\beta_0$ is independent of
the choice of vertex $v \in NW$, and in fact
$$
\beta_0 =  \log \left( \limsup_n \ (A^n_{vw})^{\frac{1}{n}} \right),
$$
for all $v,w \in NW$. We see from (\ref{r11}) that
$A^n_{vv}\xi_v \leq e^{n\beta}\xi_v$ for all $n$ when $\xi$ is an almost
$\beta$-harmonic vector. Since $\xi_v > 0$ by Lemma \ref{b16} it
follows that there can not be an almost $\beta$-harmonic vector for
$A$ unless $\beta \geq \beta_0$. We will therefore also in the
following assume that
\begin{enumerate}
\item[iv)] $\beta_0 =\log \left( \limsup_n \ (A^n_{vv})^{\frac{1}{n}}
  \right) \ < \ \infty$ for all $v\in NW$,
\end{enumerate} 
when $NW \neq \emptyset$.

\section{The recurrent case}

In this section we consider the case where $NW \neq \emptyset$ and
where 
\begin{equation}\label{r19}
\sum_{n=0}^{\infty} A^n_{vv}e^{-n\beta_0}  = \infty 
\end{equation}
for one and hence all $v\in NW$. We say that $A$ is \emph{recurrent}
in this case. The main result will be the following theorem. In the
irreducible case, i.e. when $G$ is strongly connected, it is
contained in Corollary 2 on page 371 in \cite{V}.

\begin{thm}\label{r20} Assume that $A$ is recurrent. There is a
  non-zero almost
  $\beta_0$-harmonic vector which is unique up to multiplication by
  scalars, and it is $\beta_0$-harmonic for $A$.
\end{thm}

The proof will require some preparations, some of which will also play a
role in the following sections. Let $\beta \geq \beta_0$. If the sum
\begin{equation}\label{april1}
\sum_{n=0}^{\infty} A^n_{uu} e^{-n\beta} 
\end{equation}
is finite for some $u \in NW$, it will be finite for all $u \in NW$
because $G^{NW}$ is strongly connected by Lemma \ref{b71}. Since
$$
\left\{ v \in V: \ \sum_{n=0}^{\infty} A^n_{vu} e^{-n\beta} < \infty
\right\}
$$
is hereditary and saturated, we conclude that if the sum
(\ref{april1}) is finite for some $u \in NW$, the sums
\begin{equation}\label{april2}
 \sum_{n=0}^{\infty} A^n_{vu} e^{-n\beta},
\end{equation}
where $v \in V, u \in NW$, will all be finite. Since $V_{\infty}  \subseteq
NW$ by Lemma \ref{r12}, it follows that the sums (\ref{april2}) are
all finite when $v \in V$ and $u\in V_{\infty}$, provided the sum
(\ref{april1}) is finite for one (and hence all) $u \in NW$.

\begin{lemma}\label{r22} Let $k : V_{\infty} \to [0,\infty)$
  be a non-negative function on $V_{\infty}$ such
  that 
\begin{equation}\label{april3}
\sum_{u \in V_{\infty}} \sum_{n=0}^{\infty} e^{-n
  \beta} A^n_{uu}k_u < \infty.
\end{equation} 
It follows that
\begin{equation}\label{r24}
\hat{k}_v = \sum_{u \in V_{\infty}} \sum_{n=0}^{\infty} e^{-n
  \beta} A^n_{vu}k_u < \infty
\end{equation} 
for all $v \in V$, and that $\hat{k}$ is an almost $\beta$-harmonic vector.
\end{lemma}
\begin{proof} Straightforward.
\end{proof}  

We say that a non-negative function $k : V_{\infty} \to
[0,\infty)$ is \emph{$\beta$-summable} when (\ref{april3}) holds. When
$V_{\infty}$ is a finite set and $k \neq 0$, this condition is equivalent to the
finiteness of the sum (\ref{april1}) for any $u \in NW$.

 \begin{lemma}\label{r23} (Riesz decomposition.) Let $\psi$ an almost $\beta$-harmonic
   vector. There is a unique pair $\phi, k$, where $\phi$ is a
   $\beta$-harmonic vector and $k : V_{\infty}  \to [0,\infty)$ is
   $\beta$-summable, such that
\begin{equation}\label{april9}
\psi = \phi + \hat{k}.
\end{equation}
\end{lemma}
\begin{proof} The arguments are
  standard, cf. e.g. 6.43 on page 170 in \cite{Wo}. $\phi$ is defined
  as the limit
$$
\phi_v = \lim_{n \to \infty} \sum_{w \in V} e^{-n \beta}A^n_{vw}\psi_w, 
$$
while 
$$
k_u = \psi_u - \sum_{v \in V} e^{-\beta} A_{uv}\psi_v.
$$ 
It is then easy to see that (\ref{april9}) holds. For uniqueness,
assume that $\phi'$ is $\beta$-harmonic, that $k': V_{\infty} \to
[0,\infty)$ is $\beta$-summable and that $\psi = \phi' +
\hat{k'}$. Then
$$
\lim_{n \to \infty} \sum_{w\in V} e^{-n\beta} A^n_{vw} \hat{k'}_w = 0
$$
for all $v \in V$ and hence $\phi'_v = \lim_{n \to \infty} \sum_{w\in V}
e^{-n\beta} A^n_{vw} \psi_w = \phi_v$ for all $v$. Thus $\phi = \phi'$
and $\hat{k} = \hat{k'}$. It follows that
\begin{equation*}
\begin{split}
k_v = \hat{k}_v - \sum_{w \in V}e^{-\beta} A_{vw}\hat{k}_w =
\psi_v - \sum_{w \in V}e^{-\beta} A_{vw}\psi_w  = \hat{k'}_v - \sum_{w \in V}e^{-\beta} A_{vw}\hat{k'}_w  = k'_v
\end{split}
\end{equation*}
for all $v\in V_{\infty}$.


\end{proof}

\begin{cor}\label{r25} Assume that $A$ is recurrent. It follows that
  all almost $\beta_0$-harmonic vectors are $\beta_0$-harmonic.
\end{cor}
\begin{proof} This follows from Lemma \ref{r23} since no non-zero
  function $k : V_{\infty} \to [0,\infty)$ can be $\beta_0$-summable
  in the recurrent case.
\end{proof}

With Corollary \ref{r25} in place, the proof of Theorem \ref{r20} can
be copied from the work of Vere-Jones, \cite{V}. We introduce for $v,w \in V$ and $n = 0,1,2,
\dots$, the numbers $r_{vw}(n)$ such that $r_{vw}(0) = 0, \ r_{vw}(1)
= A_{vw}$ and
$$
r_{vw}(n+1) = \sum_{u \neq w} A_{vu}r_{uw}(n) 
$$
when $n \geq 1$.

\begin{lemma}\label{f80} (Equation (4) in \cite{V}.) Assume $NW \neq \emptyset$ and that $\beta >
  \beta_0$. Then
$$
\sum_{n=0}^{\infty} A^n_{vw}e^{-n\beta} = I_{vw} +
\left(\sum_{n=1}^{\infty} r_{vw}(n)e^{-n\beta} \right) \left(
  \sum_{n=0}^{\infty} A^n_{ww} e^{-n\beta} \right) .
$$
for all $v,w \in NW$.
\end{lemma}
\begin{proof} By using the product rule for power series the stated
  equality follows from the observation that for $n \geq 1$,
$A^n_{vw} = \sum_{s=1}^n r_{vw}(s)A^{n-s}_{ww}$.
\end{proof}

\begin{lemma}\label{b11} (Lemma 4.1 in \cite{V}.) Assume $\xi :V \to [0,\infty)$
  satisfies that
$$
\sum_{w\in V} A_{vw}\xi_w \leq e^{\beta} \xi_v
$$
for all $v$. Assume $\xi_{v_0} \neq 0$ for some vertex $v_0 \in V$. It follows that
$$
\sum_{n =1}^{\infty} r_{vv_0}(n)e^{-n\beta} \leq \frac{\xi_v}{\xi_{v_0}}
$$
for all $v$.
\end{lemma}
\begin{proof} We prove by induction in $N$ that 
\begin{equation}\label{b12}
\sum_{n=1}^N
  r_{vv_0}(n)e^{-n\beta} \leq \frac{\xi_v}{\xi_{v_0}}
\end{equation}
for all $N$ and all $v$. To start the
induction note that 
$$
\xi_v \geq e^{-\beta} \sum_{w\in V} A_{vw} \xi_w \geq
e^{-\beta} A_{vv_0}\xi_{v_0} = \xi_{v_0}
r_{vv_0}(1)e^{-\beta}.
$$ 
Assume then that (\ref{b12}) holds for all $v$. It follows
that
\begin{equation*}
\begin{split}
&\frac{\xi_v}{\xi_{v_0}} \geq e^{-\beta} \sum_{w \in V}
A_{vw}\frac{\xi_w}{\xi_{v_0}} = e^{-\beta} \left( \sum_{w \neq v_0} A_{vw}
\frac{\xi_w}{\xi_{v_0}} + A_{vv_0}\right) \\
& \geq   e^{-\beta} \sum_{n=1}^N  \sum_{w\neq v_0}
  A_{vw}r_{wv_0}(n)e^{-n\beta} + e^{-\beta} A_{vv_0} \\
&= \sum_{n=1}^N
  r_{vv_0}(n+1) e^{-(n+1)\beta} + e^{-\beta} r_{vv_0}(1) =  \sum_{n=1}^{N+1}
  r_{vv_0}(n)e^{-n\beta} .   
\end{split}
\end{equation*} 
\end{proof}

\emph{Proof of Theorem \ref{r20}:}  In view Corollary \ref{r25} we
must prove the existence and essential uniqueness of a non-zero
$\beta_0$-harmonic vector. Existence: Fix a vertex $w \in NW$. It follows from Fatou's lemma that
$$
\lim_{\beta \downarrow \beta_0} \sum_{n=0}^{\infty}
A^n_{ww}e^{-n\beta} = \infty.
 $$
Since
$$
\sum_{n=0}^{\infty} A^n_{ww}e^{-n\beta} = 1 +
\left(\sum_{n=1}^{\infty} r_{ww}(n)e^{-n\beta} \right) \left(
  \sum_{n=0}^{\infty} A^n_{ww} e^{-n\beta} \right) .
$$
for all $\beta > \beta_0$ by Lemma \ref{f80}, it follows that
$$
\lim_{\beta \downarrow \beta_0} \sum_{n=1}^{\infty} r_{ww}(n)e^{-n\beta} = 1.
$$ 
By the monotone convergence theorem this leads to the conclusion
that
\begin{equation}\label{hov}
\sum_{n=1}^{\infty} r_{ww}(n)e^{-n \beta_0} = 1.
\end{equation}
Now note that
\begin{equation}\label{i10}
\begin{split}
&\sum_{u\in V} A_{vu} \left( \sum_{n=1}^{N} r_{uw}(n)e^{-n \beta_0} \right)\\
&=  \sum_{n=1}^{N} \sum_{u \neq w} A_{vu}r_{uw}(n)e^{-n \beta_0} +
A_{vw} \sum_{n=1}^{N} r_{ww}(n) e^{-n\beta_0}\\
& = \sum_{n=1}^{N} r_{vw}(n+1)e^{-n\beta_0} + A_{vw}
\sum_{n=1}^{N} r_{ww}(n) e^{-n\beta_0} \\
& = e^{\beta_0} \sum_{n=1}^{N+1} r_{vw}(n)e^{-n\beta_0} + A_{vw} \left(
  \sum_{n=1}^{N} r_{ww}(n) e^{-n\beta_0} - 1\right) .
\end{split}
\end{equation}
It follows from (\ref{i10}) and (\ref{hov}) that
$$
\left\{v \in V: \  \sum_{n=1}^{\infty} r_{vw}(n)e^{-n\beta_0} <
  \infty \right\}
$$
is both hereditary and saturated, and hence equal to $V$ since $G$
is cofinal. By letting $N$ tend to infinity in
(\ref{i10}) we see that
$$
\xi_v = \sum_{n=1}^{\infty} r_{vw}(n)e^{-n\beta_0}
$$
defines a $\beta_0$-harmonic vector $\xi$.

Uniqueness: Let $\xi'$ be a non-zero $\beta_0$-harmonic vector for $A$
 such that
$\xi'_{w} = 1$. We must show that $\xi' = \xi$. It follows from
  Lemma \ref{b11} that $\xi'_v \ \geq \ \xi_v$ for all $v \in V$.
By comparing this to the fact that
\begin{equation*}\label{b76}
e^{n\beta_0}  = \sum_{v \in V} A^n_{wv} \xi_v  = \sum_{v \in V} A^n_{wv}\xi'_v
\end{equation*}
for all $n \in \mathbb N$, we conclude that $\xi'_v = \xi_v$ for every vertex $v \in V$ with the property
that $A^n_{wv} \neq 0$ for some $n$. In particular, $\xi$ and $\xi'$
agree on $NW$ since $w \in NW$, and $G^{NW}$is strongly connected by
Lemma \ref{b71}. As $NW$ is also hereditary by the same lemma, it
follows from Lemma \ref{b65} that $\xi' = \xi$.

\qed

\subsection{When $NW$ is finite}\label{finiteNW}

In this case there are no infinite emitters in $G$ and hence all almost
$\beta$-harmonic vectors are $\beta$-harmonic.

\begin{lemma}\label{r28} Assume that $NW$ is non-empty but
  finite. Then $A$ is recurrent and there are no non-zero $\beta$-harmonic vectors for
  $A$ when $\beta > \beta_0$.
\end{lemma}
\begin{proof} Note that $e^{\beta_0}$ is the spectral radius of
  $A|_{NW}$. It follows from linear algebra that there is a non-zero vector $\psi : NW \to
  \mathbb C$ and a complex number $\lambda \in \mathbb C, \ |\lambda|
  =1$, such that $\sum_{w \in NW} A^n_{vw}\psi_w = e^{n \beta_0}
  \lambda^n \psi_v$ for all $n \in \mathbb N$ and all $v \in NW$. In particular,
  the sequence $A^n_{vw}e^{-n \beta_0}$ can not converge to zero for
  all $v,w \in NW$. It follows that $A$ must be recurrent. Assume that $\xi$ is a non-zero $\beta$-harmonic vector. Let $v \in
NW$. Then
$$
e^{n\beta} \xi_v = \sum_{w \in NW} A^n_{vw} \xi_w \leq K\max_{w \in NW} A^n_{vw} ,
$$
where $K = (\# NW) (\max_{w \in NW} \xi_w)$. There is therefore a
vertex $w \in NW$ such that
$$
e^{n_i\beta} \xi_v \leq K A^{n_i}_{vw}
$$
for an increasing sequence $n_1 < n_2 < n_3 < \dots$ of natural
numbers. It follows that 
$$
e^{\beta} = \lim_i \left(e^{n_i \beta}\xi_v\right)^{\frac{1}{n_i}}
\leq \limsup_n \left(K A^n_{vw}\right)^{\frac{1}{n}} = e^{\beta_0} ,
$$
proving that $\beta \leq \beta_0$. 
\end{proof}

\begin{cor}\label{r29} Assume that $NW$ is non-empty and finite. It
  follows that there are no non-zero $\beta$-harmonic vectors for $A$
  unless $\beta = \beta_0$.  There is a non-zero $\beta_0$-harmonic vector which unique up to multiplication by
  scalars. 
\end{cor}
\begin{proof} The first statement follows from Lemma \ref{r28}, and the
  second from Lemma \ref{r28} and Theorem \ref{r20}.
\end{proof}

\section{The transient case}

In this section we consider the non-recurrent cases. Specifically,
we assume that
\begin{equation}\label{r220}
\sum_{n=0}^{\infty} A^n_{vw}
e^{-n \beta} < \infty 
\end{equation}
for all $v,w \in V$, and refer to this as the
\emph{transient} case. The following lemma shows that the transient case
covers all the non-recurrent cases. 


\begin{lemma}\label{r30} Assume that $NW = \emptyset$ or that $\sum_{n=0}^{\infty} A^n_{uu}
e^{-n \beta} < \infty$ for some $u \in NW$. It follows that
(\ref{r220}) holds for all $v,w \in V$.
\end{lemma}
\begin{proof} If (\ref{r220}) fails for some $v,w$, it follows from
  Lemma \ref{i5} that $\sum_{n=0}^{\infty} A^n_{ww}
e^{-n \beta} = \infty$. In particular, $w \in NW$ and $\beta =
\beta_0$. This is a recurrent case, contrary to assumption.
\end{proof}

We say that $A$ is row-finite when there are no infinite emitter in $G$,
i.e. when 
$$
\# \left\{w \in V: \ A_{vw} > 0 \right\} < \infty
$$ 
for all
$v \in V$.

\begin{thm}\label{r33} Assume that (\ref{r220}) holds for all $v,w \in
  V$, and let $\beta \in \mathbb R$.

\begin{enumerate}
\item[a)] Assume that $NW = \emptyset$. Then $A$ is row-finite and
  there is a non-zero $\beta$-harmonic vector for $A$.
\item[b)] Assume that $NW$ is non-empty but finite. There are no
  non-zero almost $\beta$-harmonic vector for $A$. 
\item[c)] Assume that $NW$ is
  infinite. There is a
non-zero almost $\beta$-harmonic vector if and only if $\beta \geq
\beta_0$. 
\end{enumerate}
\end{thm}
\begin{proof} In case a), it follows from Lemma \ref{r12} that $A$ is
  row-finite. Case b) follows from Lemma \ref{r28} and Corollary
  \ref{r29}. It remains therefore only to show that there is a non-zero almost
  $\beta$-harmonic vector in case a) and c). To this end, fix $v_0
  \in V$ and set
$$
H_{v_0} = \left\{ w \in V : \ A^l_{v_0w} \neq 0 \ \text{for some} \ l
  \in \mathbb N \right\} .
$$
Consider a vertex $v \in H_{v_0}$ and choose $l \in \mathbb N$ such
that $A^l_{v_0v} \neq 0$. Then
\begin{equation}\label{i7}
\begin{split}
&A^l_{v_0v}\sum_{n=0}^{\infty} A^n_{vw}e^{-n\beta} \leq  \sum_{n=0}^{\infty} A^{l+n}_{v_0w}e^{-n\beta}   \leq e^{l\beta} \sum_{n=0}^{\infty} A^{n}_{v_0 w} e^{-n\beta}  .
\end{split}
\end{equation}
It follows that
\begin{equation}\label{i8}
\frac{\sum_{n=0}^{\infty} A^n_{vw}e^{-n\beta}}{\sum_{n=0}^{\infty}
  A^n_{v_0w}e^{-n\beta}} \leq \frac{e^{l\beta}}{A^l_{v_0v}} 
\end{equation}
for all $w \in H_{v_0}$. Note that $H_{v_0}$ is infinite. When $NW$ is
infinite this follows since $NW \subseteq H_{v_0}$ by Lemma \ref{r1}. When
$NW = \emptyset$ it follows because there are no sinks by
assumption. It follows therefore from
(\ref{i8}) that there is a sequence $\{w_k\}$ of distinct elements in $H_{v_0}$ such that the limit
$$
\eta_v = \lim_{k\to \infty} \frac{\sum_{n=0}^{\infty} A^n_{vw_k}e^{-n\beta}}{\sum_{n=0}^{\infty}
  A^n_{v_0w_k}e^{-n\beta}} 
$$
exists for all $v \in H_{v_0}$. Note that $\eta_{v_0} = 1$. By letting $N$ tend to $\infty$ in
(\ref{i2}) we find that
\begin{equation}\label{f54}
\sum_{u\in V} A_{vu} \sum_{n=0}^{\infty} A^n_{uw} e^{-n\beta} =
e^{\beta}
\sum_{n=0}^{\infty}  A^{n}_{vw}  e^{-n\beta}  - e^{\beta}I_{vw}.
\end{equation}
It follows from (\ref{f54}) that $\sum_{u \in V} A_{vu} \eta_u = e^{\beta}\eta_v$ for all $v
\in H_{v_0} \backslash V_{\infty}$, while Fatou's lemma shows that
$\sum_{u \in V} A_{vu} \eta_u \leq e^{\beta}\eta_v$ for all $v
\in H_{v_0}$. The existence of a non-zero almost ${\beta}$-harmonic vector
for $A$
follows then from Lemma \ref{b65}. 
\end{proof}

When $G$ is strongly connected and $A$ is row-finite, c) in Theorem
\ref{r33} is a result of Pruitt, \cite{P}. When $A$ is not row-finite it can happen, also when $G$ is
strongly connected, that there are no
non-zero $\beta$-harmonic vectors for any $\beta \geq\beta_0$ or that
they exist for some $\beta \geq \beta_0$ and not for others. See
\cite{Th} for such examples.

\subsection{The structure of the positive eigenvectors}\label{pos evec}

We denote the set of almost $\beta$-harmonic vectors for $A$ by
$E(A,\beta)$. Assume that $E(A,\beta) \neq 0$. For a given vertex
$v_0 \in V$ we set
$$
E(A,\beta)_{v_0} = \left\{ \xi \in E(A,\beta): \ \xi_{v_0} = 1 \right\}.
$$
Equipped with the product topology $\mathbb R^V$ is a locally convex
real vector space, and $E(A,\beta)$ is a closed convex cone in
$\mathbb R^V$. It follows from Lemma \ref{b16} that $E(A,\beta)_{v_0}$
is a base for $E(A,\beta)$, and we aim now to show that
$E(A,\beta)_{v_0}$ is a compact Choquet simplex and to obtain an integral representation of the elements
in $E(A,\beta)_{v_0}$, analogous to the Poisson-Martin integral
representation for the harmonic functions of a countable state Markov chain,
cf. e.g. Theorem 7.45 in \cite{Wo}. For this purpose we consider the partial ordering $\geq$ in
$E(A,\beta)$ defined such that $\xi \geq \mu \ \Leftrightarrow \ \xi -
\mu \in E(A,\beta)$.

\begin{lemma}\label{i19} $E(A,\beta)$
is a lattice cone; i.e. every pair of elements $\xi, \eta \in E(A,\beta)$ have a least upper bound $\xi \vee \eta \in E(A,\beta)$ and a greatest lower bound $\xi \wedge \eta \in E(A,\beta)$ for the order $\geq$.
\end{lemma}
\begin{proof} To find the greatest lower bound $\xi \wedge \mu$ of $\xi$ and $\mu$, set $\nu_v = \min \{\xi_v,\mu_v\}$. Then $\sum_{w\in V} e^{-\beta} A_{vw}\nu_w \leq \sum_{w\in V} e^{-\beta} A_{vw} \xi_w
= \xi_v$ and $\sum_{w\in V} e^{-\beta} A_{vw} \nu_w \leq \sum_{w\in V} e^{-\beta} A_{vw} \mu_w
=\mu_v,$ 
proving that 
$$
\sum_{w\in V} e^{-\beta} A_{vw} \nu_w \leq \nu_v .
$$
It follows by iteration that
$$
\sum_{w\in V} e^{-(n+1)\beta} A^{n+1}_{vw} \nu_w \leq \sum_{w\in V} e^{-n \beta} A^n_{vw} \nu_w \leq \nu_v
$$
for all $v$ and all $n$. We can therefore consider the limit
$$
\psi_v = \lim_{n \to \infty} \sum_{w\in V} e^{-n\beta} A^n_{vw} \nu_w ,
$$
and observe that $\psi \in E(A,\beta)$ while $\psi \leq \mu, \ \psi \leq
\xi$, i.e. $\psi$ is a lower bound for $\xi$ and $\mu$ in
$E(A,\beta)$. To see that it is the greatest such, consider $\varphi
\in E(A,\beta)$ such that $\varphi \leq \mu, \ \varphi
\leq \xi$. Then $\varphi \leq \nu$ and
$$
\varphi_v = \sum_{w \in V} e^{-n \beta} A^n_{vw} \varphi_w \leq
\sum_{w \in V}
e^{-n\beta} A^n_{vw} \nu_w ,
$$
for all $v,n$, and hence $\varphi \leq \psi$. This proves that $\psi$
is the greatest lower bound for $\xi$ and $\mu$, i.e. $\psi = \xi
\wedge \mu$.

The least upper bound $\xi \vee \mu$ is then given by 
$$
\xi \vee \mu = \xi + \mu - \xi \wedge \mu .
$$
Indeed, $\xi \vee \mu$ is clearly an upper bound and if $\psi$ is
another such, we find that $\xi + \mu \leq (\psi + \mu) \wedge
(\psi + \xi) = \psi + \xi \wedge \mu$ and hence $\xi \vee \mu \leq
\psi$. This shows that $\xi \vee \mu$ is the least upper bound, as claimed.
\end{proof}

Fix now a vertex $v_0 \in V$. Set 
$$
H_{v_0} = \left\{ w \in V : \ A^n_{v_0w} \neq 0 \ \text{for some} \ n
  \in \mathbb N \right\} ,
$$
and note that $H_{v_0}$ is a hereditary set of vertexes. For every $v
\in V$ we consider the function $K^{\beta}_v : H_{v_0} \to [0,\infty[$
defined by
$$
K^{\beta}_v(w) = \frac{\sum_{n=0}^{\infty} A^n_{vw} e^{-n
    \beta}}{\sum_{n=0}^{\infty} A^n_{v_0w} e^{-n \beta}} .
$$

\begin{lemma}\label{i20} Let $H \subseteq V$ be a non-empty hereditary subset of
  vertexes. For each $v \in V$ there is a $m_v
  \in \mathbb N$ such that 
$$
A^l_{vw} \neq 0 , \ l \geq m_v   \ \Rightarrow \ w \in H.
$$
\end{lemma}
\begin{proof} Define subsets $H_i \subseteq V$ recursively
  such that $H_0 = H$ and
$$
H_{n+1} = \left\{ v \in V : \ A_{vw} \neq 0 \ \Rightarrow \ w \in H_n
\right\} \cup H_n.
$$
Then $H_0 \subseteq H_1 \subseteq H_2
\subseteq \cdots$ is a sequence of hereditary subsets, and the
union $\bigcup_n H_n$ is both
hereditary and saturated. It is therefore all of $V$ since $G$ is
cofinal by assumption. When $v \in H_k$ we can use $m_v = k$. 
\end{proof}


\begin{lemma}\label{i21} Let $\beta \in \mathbb R$ and assume that
  $\sum_{n=0}^{\infty} A^n_{vw}e^{-n\beta} < \infty$ for all $v,w \in
  V$. For every vertex $v \in V$ there are positive
  numbers
  $l_v,L_v$ and a finite set $F_v \subseteq H_{v_0}$ such that
  $K^{\beta}_v(w) \leq L_v$ for all $w \in H_{v_0}$ and $0 < l_v \leq
  K^{\beta}_v(w)$ for all $w \in H_{v_0} \backslash F_v$.
\end{lemma}
\begin{proof} Consider first a vertex $v \in H_{v_0}$. There is an $l \in \mathbb N$ such
that $A^l_{v_0v} \neq 0$. Set $N_v =
{\left(A^l_{v_0v}\right)}^{-1}e^{l\beta}$ and note that the
calculation (\ref{i7}) gives the upper bound
\begin{equation}\label{f60}
K^{\beta}_v(w) \leq N_v
\end{equation}
for all $w \in H_{v_0}$. Consider then a vertex $v \in V \backslash H_{v_0}$. By
Lemma \ref{i20} there is an $m_v \in \mathbb N$ such that every path
in $G$
of length $m_v$ emitted from $v$ terminates in $H_{v_0}$. Let $\Gamma$
denote the set of finite paths $\mu$ in $G$ starting at $v$ and terminating in $H_{v_0}$,
and such that $r(\mu)$ is the only vertex in $\mu$ which is in
$H_{v_0}$. Then $|\mu| \leq m_v$ for all $\mu \in \Gamma$. Now note
that
\begin{equation*}\label{r421}
V_{\infty} \subseteq NW \subseteq H_{v_0},
\end{equation*}
where the first inclusion comes from Lemma \ref{r12} and the second
follows from Lemma \ref{r1}. Since $V_{\infty} \subseteq H_{v_0}$ and
$v \notin H_{v_0}$, the set $\Gamma$ has only finitely many elements. For every $\mu = e_1e_2 \cdots e_{|\mu|} \in \Gamma$, set
$$
W(\mu) = A_{s(e_1)r(e_1)}A_{s(e_2)r(e_2)} \cdots
A_{s(e_{|\mu|})r(e_{|\mu|})} e^{-|\mu|\beta} .
$$
Then
\begin{equation}\label{f61}
\begin{split}
&\sum_{n=0}^{\infty} A^n_{vw} e^{-n\beta}  =\sum_{\mu \in \Gamma}
W(\mu)\  \sum_{n \geq |\mu|} A^{n-|\mu|} _{r(\mu) w} e^{-(n-|\mu|)\beta}  \\
&  = \sum_{\mu \in \Gamma}
W(\mu)\  K^{\beta}_{r(\mu)}(w)  \left(
  \sum_{n=0}^{\infty} A^n_{v_0w} e^{-n\beta} \right)  \\
&  \leq  \left(
  \sum_{n=0}^{\infty} A^n_{v_0w} e^{-n\beta} \right) \sum_{\mu \in \Gamma}
W(\mu)\  N_{r(\mu)}  
\end{split}
\end{equation}
for all $w \in H_{v_0}$.
It follows that we can use $L_v = N_v$ when $v \in H_{v_0}$ and $L_v =  \sum_{\mu \in \Gamma}
W(\mu)\  N_{r(\mu)}$, otherwise.

To establish the existence of $l_v$ and $F_v$, assume for a
contradiction that for all $\epsilon > 0$ there are
infinitely many elements $w \in H_{v_0}$ such that
$$
K^{\beta}_v(w) \leq \epsilon .
$$
We can then construct a sequence $\{w_k\}$ of distinct elements in
$H_{v_0}$ such that
\begin{equation}\label{k1}
K^{\beta}_v(w_k) \leq \frac{1}{k} 
\end{equation}
for all $k$. The calculation
(\ref{f54}) shows that
\begin{equation}\label{f64}
\sum_{u\in V} A_{v'u}K^{\beta}_u(w_k) = e^{\beta}K^{\beta}_{v'}(w_k) -  e^{\beta}\left(
  \sum_{n=0}^{\infty} A^n_{v_0w_k}e^{-n\beta}\right)^{-1} I_{v'w_k} 
\end{equation}
for all $v'\in V$ and all $k \in \mathbb N$. It follows from (\ref{f64}) that a condensation
point $\xi =
\left(\xi_u\right)_{u \in V}$ in  $\prod_{u \in V} \left[0,L_u\right]$
of the sequence
$$
\left( K_u^{\beta}(w_k) \right)_{u \in
  V}, \ k \in \mathbb N,
$$
is an almost $\beta$-harmonic vector for $A$ with $\xi_{v_0} =
1$. But (\ref{k1}) implies that $\xi_v= 0$ which is
impossible by Lemma \ref{b16}. This contradiction shows that there must be an $l_v > 0$
and a finite set $F_v \subseteq H_{v_0}$ such that $l_v \leq
K^{\beta}_v(w)$ for all $w \in H_{v_0} \backslash F_v$.

\end{proof}

It follows from Lemma \ref{i21}
that $K^{\beta}_v$ is a bounded function on $H_{v_0}$,
i.e. $K^{\beta}_v \in l^{\infty}\left(H_{v_0}\right)$ for all $v \in V$. We denote by $1_{w}$ the characteristic function of an
element $w \in H_{v_0}$. Let $\mathcal A_{\beta}$ be the $C^*$-subalgebra of
$l^{\infty}\left(H_{v_0}\right)$ generated by $K^{\beta}_v,  v \in V$,
and the functions $1_w,  w \in H_{v_0}$, and let $\mathcal B_{\beta}$ be the image of $\mathcal A_{\beta}$ in the
quotient algebra
$$
l^{\infty}\left(H_{v_0}\right)/c_0\left(H_{v_0} \backslash
  V_{\infty}\right). 
$$
Here $c_0\left(H_{v_0} \backslash
  V_{\infty}\right)$ denotes the ideal in $l^{\infty}\left(H_{v_0}\right)$
  consisting of the elements $f : H_{v_0} \to \mathbb C$ with the
  property that for all $\epsilon > 0$ there is a finite subset $F
  \subseteq H_{v_0} \backslash V_{\infty}$ such that $|f(w)| \leq
  \epsilon \ \forall w \in H_{v_0} \backslash F$. In particular,
  $f(V_{\infty}) = \{0\}$ when $V_{\infty} \neq \emptyset$ and $f \in c_0\left(H_{v_0} \backslash
  V_{\infty}\right)$.

Note that it
follows from Lemma \ref{i21} that for every $v \in V$ there is a
finite subset $F_v \subseteq H_{v_0}$ such that 
$$
K^{\beta}_v \ + \sum_{w
  \in F_v} 1_w
$$ 
is invertible in
$l^{\infty}(H_{v_0})$. Thus $\mathcal A_{\beta}$ and $\mathcal B_{\beta}$ are both
unital $C^*$-algebras. Since they are also separable, the set
$X_{\beta}$ of characters of $\mathcal B_{\beta}$ is a compact
metric space and $\mathcal B_{\beta}$ can be identified with
$C\left(X_{\beta}\right)$ via the Gelfand transform. Since evaluation
at a vertex $w \in V_{\infty}$ annihilates $c_0\left(H_{v_0}
  \backslash V_{\infty}\right)$, each element of $V_{\infty}$ gives
rise to character on $\mathcal B_{\beta}$ and hence an element of
$X_{\beta}$. It follows that there is a canonical inclusion
\begin{equation}\label{r206}
V_{\infty}   \subseteq  X_{\beta}.
\end{equation}
For each $v\in V$ the function $w \mapsto K_v^{\beta}(w)$ is
an element of $\mathcal A_{\beta}$ and its image in $\mathcal B_{\beta}$ is a continuous
function on $X_{\beta}$ which we also denote by
$K_v^{\beta}$. Let $M\left(X_{\beta}\right)$ denote the set of Borel probability measures on
$X_{\beta}$.  We consider $M\left(X_{\beta}\right)$ as a compact convex
set in the weak*-topology obtained by considering the measures
as elements of the dual of $C(X_{\beta})$.


\begin{thm}\label{Martin} Let $\beta \in \mathbb R$. Assume that $NW =
\emptyset$ or that $NW$ is infinite. Assume also that $\sum_{n=0}^{\infty} A^n_{vw} e^{-n\beta} < \infty$
for all $v,w \in V$. Then $E(A,\beta)_{v_0}$ is a non-empty compact metrizable Choquet
  simplex and there is a
  continuous affine surjection $I : M(X_{\beta}) \to
  E(A,\beta)_{v_0}$ defined such that
$$
I(m)_v = \int_{X_{\beta}} K^{\beta}_v \ dm .
$$
\end{thm}  
\begin{proof} 
It follows from Theorem \ref{r33} that $E(A,\beta)_{v_0} \neq \emptyset$. 
Since $E(A,\beta)_{v_0}$ is a base of $E(A,\beta)$, which is a lattice cone by Lemma \ref{i19}, to conclude that
  $E(A,\beta)_{v_0}$ is a compact Choquet simplex we need only show
  that $E(A,\beta)_{v_0}$ is compact in $\mathbb R^V$,
  cf. e.g. \cite{BR}. Note that $I : M(X_{\beta}) \to \mathbb R^V$ is continuous by definition of
  the topologies. It suffices therefore to show that
 \begin{equation}\label{f65}
I\left(M(X_{\beta})\right) = E(A,\beta)_{v_0} .
\end{equation}
Let $x\in
  X_{\beta}$. If $x \in V_{\infty}$, we find that 
$$
K^{\beta}_v(x) =
\frac{\sum_{n=0}^{\infty} A^n_{vx}e^{-n\beta}}{\sum_{n=0}^{\infty}
  A^n_{v_0x}e^{-n\beta}}
$$ 
and it follows from (\ref{f54}) that 
\begin{equation}\label{s110}
\sum_{u\in V} A_{vu}K^{\beta}_u(x) \leq e^{\beta} K^{\beta}_v(x)
\end{equation}
for all $v \in V$ with equality when $v \neq x$.
To show that (\ref{s110}) also holds when $x \in X_{\beta} \backslash V_{\infty}$, note first that point-evaluations at points in $H_{v_0}$ constitute a dense subset of the 
character space of $\mathcal A_{\beta}$. It follows that there is a sequence $\{u_k\}
\subseteq H_{v_0}$ such that $K^{\beta}_v(x) = \lim_{k \to \infty}
K^{\beta}_v(u_k)$ for all $v \in V$. Since $\mathcal B_{\beta}$ is the
quotient of $\mathcal A_{\beta}$ by the ideal $c_0(H_{v_0} \backslash
V_{\infty})$, the sequence $\{u_k\}$ must eventually leave every finite
subset of $H_{v_0} \backslash V_{\infty}$, and since $x \notin
V_{\infty}$ it must also eventually leave every finite
subset of $V_{\infty}$. That is, $\lim_{k \to \infty} u_k
= \infty$ in the sense that for any finite set $F$ of vertexes there is an
$N\in \mathbb N$ such that $u_k \notin F \ \forall k \geq N$. It follows then from (\ref{f54}) that (\ref{s110})
holds for all $v \in V$ with equality when $v \notin V_{\infty}$. This
shows that the inequality
$$
\sum_{u\in V} A_{vu}K^{\beta}_u \leq e^{\beta} K^{\beta}_v
$$
holds point-wise on $X_{\beta}$ for all $v \in V$, and that equality
holds globally on $X_{\beta}$  when
$v\notin V_{\infty}$. It follows
  therefore by integration that $I(m) \in E(A,\beta)_{v_0}$ for all
  $m\in M(X_{\beta})$.

To obtain (\ref{f65}) it remains to show that
$E(A,\beta)_{v_0} \subseteq I(M(X_{\beta}))$. For this we modify the argument from the proof of Theorem 4.1 in \cite{Sa}. Let
$\xi \in E(A,\beta)_{v_0}$. Fix an element $x \in X_{\beta}$ and let
$\{w_k\}$ be a sequence of elements in $H_{v_0}$ such that $\lim_{k \to\infty} K^{\beta}_v(w_k) =
K^{\beta}_v(x)$ for all $v \in V$. Observe that since $K^{\beta}_v(x)
> 0$ by Lemma \ref{b16} it follows that
\begin{equation}\label{l2}
\lim_{n \to \infty} nK^{\beta}_v(w_n) = \infty
\end{equation}
for all $v \in V$. For each $n \in \mathbb N$, set
$$
\xi^n_v = \min \left\{ \xi_v , \ n  K^{\beta}_v(w_n) \right\} .
$$
Then
\begin{equation}\label{f52}
\lim_{m \to \infty} e^{-m\beta} \sum_{w \in V} A^m_{vw}\xi^n_w = 0
\end{equation}
for all $n,v$. To see this note that
\begin{equation*}
\begin{split}
&\left(\sum_{j=0}^{\infty} A^j_{v_0w_n} e^{-j \beta}\right)
e^{-m\beta} \sum_{w \in V} A^m_{vw}\xi^n_w \ \leq \ \left(\sum_{j=0}^{\infty}
  A^j_{v_0w_n} e^{-j \beta}\right) n e^{-m\beta} \sum_{w \in V} A^m_{vw}
K^{\beta}_w(w_n)\\
& = ne^{-m\beta} \sum_{w \in V} A^m_{vw}
\sum_{j=0}^{\infty}A^j_{ww_n}e^{-j\beta} 
= n\sum_{j
  \geq m} A^j_{vw_n}e^{-j\beta} .  
\end{split}
\end{equation*}
Hence (\ref{f52}) follows because $\sum_{j=0}^{\infty}
A^j_{vw_n}e^{-j\beta} < \infty$. Set
$$
k_n(v) = \xi^n_v - e^{-\beta} \sum_{w \in V} A_{vw} \xi^n_w .
$$
We claim that $k_n \geq 0$. To see this observe first that it follows
from (\ref{f64}) that $e^{-\beta}\sum_{w \in V} A_{vw}
K^{\beta}_w(w_n) \leq K^{\beta}_v(w_n)$. Combined with
$e^{-\beta}\sum_{w\in V}
A_{vw}\xi_w  \leq \xi_v$ this 
implies that
$$
e^{-\beta}\sum_{w \in V} A_{vw}\xi^n_w \ \leq \ \min \left\{ \xi_v , \
  n K^{\beta}_v(w_n)\right\} \ = \
\xi^n_v ,
$$
proving the claim. Since
$$
\sum_{l =0}^m e^{-l\beta} \sum_{w \in V} A^l_{vw}k_n(w) = \xi^n_v -
e^{-(m+1)\beta} \sum_{w\in V} A^{m+1}_{vw}\xi^n_w ,
$$
it follows from (\ref{f52}) that
\begin{equation}\label{b57}
\xi^n_v = \sum_{l =0}^{\infty} e^{-l\beta} \sum_{w\in V} A^l_{vw}k_n(w) =
\sum_{w\in H_{v_0}} K^{\beta}_v(w) h_n(w)
\end{equation}
when $v \in H_{v_0}$, where $h_n(w) = \sum_{l=0}^{\infty} e^{-l\beta}
A^l_{v_0w}k_n(w)$. In particular, it follows from (\ref{b57}) that
$$
\sum_{w \in H_{v_0}} h_n(w) = \sum_{w\in H_{v_0}} K^{\beta}_{v_0}(w) h_n(w) =
\xi^n_{v_0} \leq \xi_{v_0} = 1 .
$$  
We can therefore define a positive linear functional $\mu_n$ of norm
$\leq 1$ on $\mathcal A_{\beta}$ 
such that
$$
\mu_n(g) = \sum_{w \in H_{v_0}} g(w)h_n(w) .
$$
By compactness of the unit ball in the dual space of $\mathcal A_{\beta}$ there
is a strictly increasing sequence $\{n_l\}$ in $\mathbb N$ and a
positive linear functional $\mu$ on $\mathcal A_{\beta}$ such that 
$$
\mu(g) = \lim_{l \to \infty} \mu_{n_l}(g) 
$$
for all $g \in \mathcal A_{\beta}$. Since $\lim_{l \to \infty} \xi^{n_l}_v =
\xi_v$ by (\ref{l2}), it follows from (\ref{b57}) that 
\begin{equation}\label{f66}
\mu\left(K^{\beta}_v\right) = \xi_v 
\end{equation}
for all $v \in H_{v_0}$. For any fixed $w \in H_{v_0}$ we
have that $\xi^{n_l}_w = \xi_w$ for all large $l$, and hence also that
$k_{n_l}(w) = 0$ for all large $l$ when $w \in H_{v_0} \backslash V_{\infty}$. It follows that $\lim_{l \to \infty}
\mu_{n_l}(1_w) = 0$ for all $w \in H_{v_0} \backslash V_{\infty}$, which shows that $\mu$
factors through $\mathcal B_{\beta}$. It follows therefore from (\ref{f66}) and
the Riesz representation theorem that there is a Borel probability
measure $m$ on $X_{\beta}$ such that $I(m)_v = \xi_v$ for all $v\in
H_{v_0}$. By Lemma \ref{b65} this implies that $I(m) = \xi$.
\end{proof}

When $A$ is sub-stochastic, irreducible and $\beta = 0$, it follows from
the abstract characterisation given in Theorem 7.13 of \cite{Wo} that
the spectrum of $\mathcal A_{\beta}$ is the Martin compactification of the
associated Markov chain, cf. Definition 7.17 in \cite{Wo}, while
$X_{\beta}$ is the Martin boundary when $A$ is also row-finite (has
finite range in the sense of \cite{Wo}). When $A$ is not row-finite
$X_{\beta}$ consists of the Martin boundary and the vertexes
$V_{\infty}$.

Let $\delta_x$ denote the Dirac measure at a point $x \in
X_{\beta}$. Then 
$$
I(\delta_x)_v = K^{\beta}_v(x)
$$
for all $v \in V$. Thus $I$ takes the extreme points in $M(X_{\beta})$
to elements of the form $v \mapsto K^{\beta}_v(x)$ for some $x \in
X_{\beta}$. By definition of $X_{\beta}$, when $x \in X_{\beta}
\backslash V_{\infty}$, there is a sequence
$\{w_k\}$ of distinct elements in $H_{v_0}$ such that
$$
\lim_{k \to \infty} K^{\beta}_v(w_k) = K^{\beta}_v(x) .
$$
Recall now that under a continuous affine surjection between compact
convex sets, the pre-image of an extremal point is a closed face and
therefore contains an extremal point. In this way we obtain from Theorem
\ref{Martin} the following description of the extreme points in $E(A,\beta)_{v_0}$.

\begin{cor}\label{add} Let $\beta \in \mathbb R$.  Assume that $NW =
\emptyset$ or that $NW$ is infinite. Assume also that $\sum_{n=0}^{\infty} A^n_{vw} e^{-n\beta} < \infty$
for all $v,w \in V$. Let $\xi$ be an extremal point of
  $E(A,\beta)_{v_0}$. There is a sequence $\{w_k\}$ of distinct elements in $H_{v_0}$ such
that
\begin{equation}\label{s123}
\xi_v = \lim_{k \to \infty} \frac{\sum_{n=0}^{\infty}
  A^n_{vw_k}e^{-n\beta}}{\sum_{n=0}^{\infty} A^n_{v_0w_k} e^{-n\beta}}
\end{equation}
for all $v \in V$, or there is a vertex $w \in H_{v_0} \cap
V_{\infty}$ such that
\begin{equation}\label{r202}
\xi_v = \frac{\sum_{n=0}^{\infty}
  A^n_{vw}e^{-n\beta}}{\sum_{n=0}^{\infty} A^n_{v_0w} e^{-n\beta}}
\end{equation}
for all $v \in V$.
\end{cor}



Every infinite emitter $w \in V_{\infty}$ gives rise to an
extremal element in $E(A,\beta)_{v_0}$ via the formula
(\ref{r202}) (in the transient case we consider here). This follows
from the uniqueness part in the Riesz decomposition lemma, Lemma \ref{r23}. The question about which
sequences of vertexes $\{w_k\}$ give rise to extremal $\beta$-harmonic
by the formula
(\ref{s123}) is generally much more difficult to
answer. But at least we shall show below that the vertexes
in $\{w_k\}$ can be chosen to be the vertexes in an infinite path in
$G$ without repeated vertexes. It follows from this that there are cases
where $A$ is not row-finite and where there are no non-zero
$\beta$-harmonic vectors, for any $\beta$, because there are no paths in
$G$ of this sort. This can occur also when all our
standing assumptions hold and $G$ is strongly connected. See Example
2.9 in \cite{Ru}.

 Let $\partial
E(A,\beta)_{v_0}$ be the extreme boundary of the Choquet simplex
$E(A,\beta)_{v_0}$. By identifying an element $x \in X_{\beta}$ with
the corresponding Dirac measure $\delta_x \in M(X_{\beta})$, we have
an inclusion $X_{\beta} \subseteq M(X_{\beta})$, and we set
$$
\partial X_{\beta} = X_{\beta}\cap I^{-1}(E(A,\beta)_{v_0}) ,
$$
which is a Borel subset of $X_{\beta}$, cf. Theorem 4.1.11 in
\cite{BR}. Thus $\partial X_{\beta}$ consists of the elements
$x\in X_{\beta}$ for which $v \mapsto K^{\beta}_v(x)$ is extremal in
$E(A,\beta)_{v_0}$.

\begin{lemma}\label{april11} $V_{\infty} \subseteq \partial X_{\beta}$
  and $I$ is injective on $\partial X_{\beta}$.
\end{lemma}
\begin{proof} Let $u \in V_{\infty}$. Then
$$
I(u)_v = K^{\beta}_v(u) = \widehat{k(u)}_v,
$$
where $k(u): V_{\infty} \to [0,\infty)$ is the function
$$
k(u)(w) = \begin{cases} \left(\sum_{n=0}^{\infty} A^n_{v_0u}
    e^{-n\beta} \right)^{-1} & \ \text{when} \ w = u, \\ 0 & \ \text{otherwise} .
\end{cases}
$$ 
It follows therefore from the uniqueness part of the statement in
Lemma \ref{r23} that $V_{\infty} \subseteq \partial X_{\beta}$ and
that $I$ is injective on $V_{\infty}$. Furthermore,
$\partial E(A,\beta)_{v_0}$ is the disjoint union
$$
\partial E(A,\beta)_{v_0} = \partial H(A,\beta)_{v_0} \sqcup \left\{
  I(u) : \ u\in V_{\infty} \right\},
$$
where $\partial H(A,\beta)_{v_0}$ is the set of extremal
$\beta$-harmonic elements of $E(A,\beta)_{v_0}$. It suffices
therefore now to show that $I$ is injective on
$\partial X_{\beta}  \cap I^{-1}( \partial H(A,\beta)_{v_0})$. Since $I(u)$ is not $\beta$-harmonic when $u \in
V_{\infty}$, it follows that 
$$
X_{\beta}  \cap I^{-1}( \partial H(A,\beta)_{v_0}) \subseteq X_{\beta}
\backslash V_{\infty} .
$$
As $\mathcal B_{\beta}$ is generated by the image of the
functions $K^{\beta}_v, v \in V$, and $1_u, u \in V_{\infty}$, it follows
that $I$ is injective on $X_{\beta} \backslash V_{\infty}$ and
therefore also on $\partial X_{\beta}  \cap I^{-1}( \partial H(A,\beta)_{v_0})$. 
\end{proof}

 Let $M(\partial X_{\beta})$ denote the subset of
  $M(X_{\beta})$ consisting of the elements $m$ of $M(X_{\beta})$ with
  $m(\partial X_{\beta}) = 1$.

\begin{thm}\label{r331} The map $I : M(\partial X_{\beta}) \to
  E(A,\beta)_{v_0}$ is an affine bijection.
\end{thm}
\begin{proof} Surjectivity: Let $\psi \in E(A,\beta)_{v_0}$. Since
  $E(A,\beta)_{v_0}$ is a metrizable Choquet simplex there is a unique
  Borel probability measure $\nu$ on $\partial E(A,\beta)_{v_0}$ such
  that 
$$
\psi = \int_{\partial E(A,\beta)_{v_0}} z \ \nu(z),
$$
in the sense that
$$
a(\psi) = \int_{\partial E(A,\beta)_{v_0}} a(z) \ \nu(z),
$$
for all continuous and affine functions $a$ on $E(A,\beta)_{v_0}$,
cf. Theorem 4.1.15 in \cite{BR}. By Lemma \ref{april11} we can define a Borel probability measure $m$
on $\partial X_{\beta}$ such that $m(B) = \nu(I(B))$. Extending $m$ to a
Borel probability measure on $X_{\beta}$ such that $m\left(X_{\beta}
  \backslash \partial X_{\beta}\right) = 0$, we can consider $I(m)$. For
each $v\in V$, the evaluation map $\ev_v(\phi) = \phi_v$ is continuous
and affine on $E(A,\beta)_{v_0}$ so we find that
\begin{equation*}
\begin{split}
&I(m)_v = \int_{\partial X_{\beta}} K^{\beta}_v(y) \ d m(y) =
\int_{\partial  E(A,\beta)_{v_0}} K^{\beta}_v(I^{-1}(z)) \ d\nu (z) \\
& = \int_{\partial  E(A,\beta)_{v_0}} \ev_v \left(z\right) \ d \nu(z) =
\ev_v(\psi) = \psi_v .
\end{split}
\end{equation*}

Injectivity: Assume $m_1,m_2 \in M(\partial X_{\beta})$ and that
$I(m_1) = I(m_2)$. Using Lemma \ref{april11} again it follows that there are Borel probability measures $\nu_i$ on
$\partial E(A,\beta)_{v_0}$ such that $m_i = \nu_i \circ I$ on $\partial
X_{\beta}$, $i =1,2$. Note that
\begin{equation}
\begin{split}
& I(m_i)_v = \int_{\partial X_{\beta}} K^{\beta}_v(y) \ dm_i(y) =
\int_{\partial E(A,\beta)_{v_0}} K^{\beta}_v(I^{-1}(z)) \ d\nu_i(z) \\
&= \int_{\partial E(A,\beta)_{v_0}} \ \ev_v(z) \ d\nu_i(z) = \ev_v \left(
  \int_{\partial E(A,\beta)_{v_0}} \ z \ d\nu_i(z) \right)
\end{split}
\end{equation}
for all $v\in V,i =1,2$. As $I(m_1) =I(m_2)$ it follows that
$$
\int_{\partial E(A,\beta)_{v_0}} \ z \
d\nu_1(z) = \int_{\partial E(A,\beta)_{v_0}} \ z \
d\nu_2(z) .
$$
Since $\partial E(A,\beta)_{v_0}$ is a Choquet simplex this implies
that $\nu_1 = \nu_2$, and hence $m_1 = m_2$.

\end{proof}

\begin{cor}\label{april30}  $\partial X_{\beta} \backslash V_{\infty}
  = \left\{ x \in \partial X_{\beta} : \ \sum_{u \in V} A_{vu}
    K^{\beta}_u(x) = e^{\beta} K^{\beta}_v (x) \ \forall v \in V
  \right\}$.
\end{cor}
\begin{proof}
Let $x \in \partial X_{\beta} \backslash V_{\infty}$. Then
$I(x) \in \partial E(A,\beta)_{v_0}$ and it follows from
the Riesz decomposition, Lemma \ref{r23}, that $I(x)$ is
either $\beta$-harmonic or equal to $I(w)$ for some $w \in
V_{\infty}$. Since $x \notin V_{\infty}$ and $I$ is injective on
$M(\partial X_{\beta})$ by Theorem
\ref{r331}, it follows that $I(x)$ is $\beta$-harmonic, i.e. $
 \sum_{u \in V} A_{vu}K^{\beta}_u(x) = e^{\beta} K^{\beta}_v (x)$ for all $v \in V$.
\end{proof}

It follows from Theorem \ref{r331} and Corollary \ref{april30} that
the map $I$ of Theorem \ref{Martin} restricts to an affine bijection
between the $\beta$-harmonic elements of $E(A,\beta)_{v_0}$ and the
Borel probability measures $m$ on $X_{\beta}$ with the property that
$m\left(\partial X_{\beta} \backslash V_{\infty}\right) =1$.

\subsection{An improvement} The purpose with this section is the use
results from the theory of Markov chains to improve the description of the
extremal $\beta$-harmonic vectors given in Corollary \ref{add}.

Let $P(V)$ denote the set 
$$
P(V) = \left\{ (v_i)_{i=1}^{\infty} \in V^{\mathbb N} : \
  A_{v_iv_{i+1}} > 0 , i = 1,2,3,\cdots \right\} .
$$
Since there are no multiple edges in $G$ the elements of $P(V)$ are
the infinite paths in $G$. We consider
$P(V)$ as a complete metric space whose topology is generated by the
cylinder sets
$$
C(v_1v_2\cdots v_n) = \left\{ (x_i)_{i=1}^{\infty} \in P(V):
  \ x_i =v_i, \ i = 1,2,\cdots, n \right\}.
$$

\begin{lemma}\label{s5} Let $\psi$ be a non-zero $\beta$-harmonic
  vector. There is a
  Borel measure $m_{\psi}$ on $P(V)$ such that
$$
m_{\psi}\left(C(v_1v_2\cdots v_n)\right) = A_{v_1v_2}A_{v_2v_3}
\cdots A_{v_{n-1}v_n} e^{-(n-1)\beta} \psi_{v_n}
$$
for every cylinder set $C(v_1v_2\cdots v_n)$.
\end{lemma}
\begin{proof} The matrix
\begin{equation}\label{s102}
B_{vw} = e^{-\beta} \psi_v^{-1}A_{vw} \psi_w
\end{equation}
is stochastic. It follows then from Theorem 1.12 in
\cite{Wo} that there is a Borel probability measure $m_v$ on $C(v)$ for each
$v \in V$ such that
\begin{equation}\label{s103}
{m_v}\left(C(vv_2\cdots v_n)\right) =   \psi_v^{-1}A_{vv_2}A_{v_2v_3}
\cdots A_{v_{n-1}v_n} e^{-(n-1)\beta} \psi_{v_n}.
\end{equation}
Define $m_{\psi}$ such that
\begin{equation*}\label{s104}
m_{\psi}(B) = \sum_{v \in V} \psi_v m_v\left( C(v) \cap B\right) .
\end{equation*}
\end{proof}

\begin{thm}\label{r101} Assume that $\sum_{n=0}^{\infty} A^n_{vw} e^{-n\beta} < \infty$
for all $v,w \in V$. Let $\psi$ be an extremal non-zero
$\beta$-harmonic vector with $\psi_{v_0} =1$. There
  is an infinite path $t = (t_i)_{i=1}^{\infty}$ in $P(V)$ such that
  $t_1 = v_0$, $i \neq j \Rightarrow t_i \neq t_j$, and
\begin{equation}\label{r200}
\psi_v = \lim_{k \to \infty} \frac{\sum_{n=0}^{\infty}
  A^n_{vt_k}e^{-n\beta}}{\sum_{n=0}^{\infty} A^n_{v_0t_k}e^{-n\beta}}
\end{equation}
for all $v \in V$.
\end{thm}
\begin{proof} By construction the measure $m_v$ on $C(v)$ given by (\ref{s103}) is the measure ${\bf Pr}_v$ from
Theorem 1.12 in \cite{Wo} and the measure $P_x$, with $x = v$, from
(2.2) in \cite{Sa}, coming from the stochastic matrix
(\ref{s102}). Since $\psi$ is extremal in
$E(A,\beta)_{v_0}$ the constant vector $1$ is minimal harmonic for the stochastic matrix $B$ restricted to $H_{v_0}$. It
follows therefore from Theorem 5.1 in \cite{Sa} that with respect to
the measure $m_{v_0}$ almost all
elements in
\begin{equation}\label{s106}
\left\{(x_i) \in P(V) : \ x_1 = v_0 \right\}
\end{equation}
have the property that
\begin{equation}\label{s105}
\lim_{i \to \infty} \frac{\sum_{n=0}^{\infty}
  B^n_{vx_i}}{\sum_{n=0}^{\infty} B^n_{v_0x_i}} = 1 
\end{equation}
for all $v \in H_{v_0}$. Let $w \in H_{v_0}$. Since $\sum_{n=0}^{\infty} B_{ww}^n =
\sum_{n=0}^{\infty} A^n_{ww} e^{-n\beta} < \infty$, it
follows from Theorems 3.2 and 3.4 in \cite{Wo} that the set
$$
\left\{(x_i) \in P(V) : \ x_1 = v_0, \ x_i = w \ \text{for at most
    finitely many} \ i \right\}
$$
has full $m_{v_0}$-measure. Since $H_{v_0}$ is countable it follows that so has
$$
 \bigcap_{w \in H_{v_0}} \left\{(x_i) \in P(V) : \ x_1 = v_0, \ x_i = w \ \text{for at most
    finitely many} \ i \right\} ,
$$ 
which is the same set as $\left\{(x_i) \in P(V) : \ x_1 = v_0, \ \lim_{i \to \infty} x_i =
  \infty \right\}$. Hence the
set of elements $(x_i)$ from (\ref{s106}) for which (\ref{s105}) holds
for all $v \in H_{v_0}$ and at the same have the property that $\lim_{i
  \to \infty} x_i = \infty$, is also a set of full $m_{v_0}$-measure. Since 
$$
 \frac{\sum_{n=0}^{\infty}
  B^n_{vx_i}}{\sum_{n=0}^{\infty} B^n_{v_0x_i}} = \psi_v^{-1}  \frac{\sum_{n=0}^{\infty}
  A^n_{vx_i}e^{-n\beta}}{\sum_{n=0}^{\infty} A^n_{v_0x_i}e^{-n\beta}}  , 
$$
it follows that there is an element $(t'_i)$ from (\ref{s106}) such
that $\lim_{i \to \infty} t'_i = \infty$ and 
\begin{equation}\label{s113}
\lim_{i \to \infty}  \frac{\sum_{n=0}^{\infty}
  A^n_{vt'_i}e^{-n\beta}}{\sum_{n=0}^{\infty} A^n_{v_0t'_i}e^{-n\beta}}
= \psi_v
\end{equation}
for all $v \in H_{v_0}$. It follows then from the uniqueness part of
the statement in Lemma \ref{b65} that (\ref{s113}) holds for all $v
\in V$. Finally, since $\lim_{i \to \infty} t'_i = \infty$ it is easy
to construct from $t'$ a path $(t_i) \in P(V)$ such that $i \neq j
\Rightarrow t_i \neq t_j$, and such that there is a sequence $n_1 <
n_2 < n_3 < \cdots $ in $\mathbb N$ with $t_i = t'_{n_i}$ for all
$i$. The sequence $(t_i)$ has the stated properties.

\end{proof}

In summary we have found that in the transient case the extremal elements
of $E(A,\beta)_{v_0}$ consist of the vectors arising from an element
$w \in H_{v_0} \cap V_{\infty}$ by the formula (\ref{r202}) when they
are not $\beta$-harmonic, and are given by an infinite path $(t_i) \in
P(V)$ such that $i \neq j \Rightarrow t_i \neq t_j$ via the formula
(\ref{r200}) when they are $\beta$-harmonic.


\end{document}